\documentclass[11pt,reqno]{amsart}
\usepackage[margin=4cm]{geometry}

\usepackage{amsmath,amsfonts,amssymb,amscd,amsthm,amsbsy,amscd,bbm, epsf,calc,enumerate,color,graphicx,verbatim,cite,setspace}
\usepackage{thmtools}
\usepackage{thm-restate}
\usepackage[margin=1cm]{caption}
\usepackage{mathtools}
\usepackage[colorlinks=true, linkcolor=blue, citecolor=red, pdfborder={0 0 0}]{hyperref}
\usepackage{cleveref}
\usepackage{accents}
\usepackage{wrapfig}
\usepackage{multirow}
\usepackage{float}
\usepackage{blkarray}

\theoremstyle{plain}
 \newtheorem{theorem}{Theorem}[section]
 \newtheorem{lemma}[theorem]{Lemma}
 \newtheorem{proposition}[theorem]{Proposition}
 
 \newtheorem{conjecture}[theorem]{Conjecture}

\theoremstyle{definition}
 \newtheorem{definition}[theorem]{Definition}

\theoremstyle{remark}
 \newtheorem{remark}[theorem]{Remark}

\numberwithin{equation}{section}
\numberwithin{theorem}{section}

\newcommand\nc\newcommand
\nc\dmo\DeclareMathOperator
\nc\bs\boldsymbol

\nc{\Red}[1]{\textcolor[rgb]{0.8,0,0}{[#1]}}
\nc{\Green}[1]{\textcolor[rgb]{0,0.5,0}{[#1]}}
\nc{\Blue}[1]{\textcolor[rgb]{0,0,0.8}{[#1]}}

\nc{\abbr}[1]{{\sc{\lowercase{#1}}}}
\nc{\abbrev}[1]{{\raisebox{-0.0pt}{\tt{\uppercase{\scalebox{0.75}{#1}}}}}}

\nc{\ii}{\mathrm{i}}

\dmo{\sign}{sign}
\dmo{\spt}{spt}
\dmo{\supp}{supp}
\dmo{\sym}{Sym}
\nc{\R}{\mathbb{R}}
\nc{\C}{\mathbb{C}}
\nc{\N}{\mathbb{N}}
\nc{\Z}{\mathbb{Z}}
\nc{\erdos}{Erd\H os }
\nc{\er}{Erd\H os--R\'enyi } 
\dmo{\ls}{\lesssim}
\dmo{\gs}{\gtrsim}
\def \<{\langle}
\def \>{\rangle}

\def \lf {\lfloor}
\def \rf {\rfloor}
\nc{\expo}[1]{\exp \left( #1 \rule{0mm}{3mm}\right)}
\DeclarePairedDelimiter\parentheses{\lparen}{\rparen}

\nc{\dd}{\mathrm{d}}

\dmo{\e}{\mathbb{E}}
\dmo{\var}{Var}
\dmo{\pr}{\mathbb{P}}
\dmo{\un}{\mathbbm{1}}
\dmo{\1}{\mathbf{1}}
\nc{\eqd}{\,{\buildrel d \over =}\,}

\nc{\bad}{\mathcal{B}}
\nc{\event}{\mathcal{E}}
\nc{\good}{\mathcal{G}}
\nc{\pro}[1]{\mathbb{P}\parentheses*{#1 \rule{0mm}{0mm}}}
\nc{\set}[1]{\left\{ #1 \right\}}

\dmo{\I}{I}
\dmo{\tr}{tr}
\dmo{\rank}{rank}
\dmo{\rk}{Rank}
\dmo{\corank}{corank}
\def \tran {\mathsf{T}}

\nc{\Span}{\operatorname{span}}
\dmo{\per}{\text{per}}

\nc{\eps}{\varepsilon}
\nc{\ep}{\epsilon}

\nc{\mA}{\mathcal{A}}
\nc{\mB}{\mathcal{B}}
\nc{\mC}{\mathcal{C}}
\nc{\mD}{\mathcal{D}}
\nc{\mE}{\mathcal{E}}
\nc{\mF}{\mathcal{F}}
\nc{\mG}{\mathcal{G}}
\nc{\mH}{\mathcal{H}}
\nc{\mI}{\mathcal{I}}
\nc{\mJ}{\mathcal{J}}
\nc{\mK}{\mathcal{K}}
\nc{\mL}{\mathcal{L}}
\nc{\mM}{\mathcal{M}}
\nc{\mN}{\mathcal{N}}
\nc{\mO}{\mathcal{O}}
\nc{\mP}{\mathcal{P}}
\nc{\mQ}{\mathcal{Q}}
\nc{\mR}{\mathcal{R}}
\nc{\mS}{\mathcal{S}}
\nc{\mT}{\mathcal{T}}
\nc{\mU}{\mathcal{U}}
\nc{\mV}{\mathcal{V}}
\nc{\mW}{\mathcal{W}}
\nc{\mX}{\mathcal{X}}
\nc{\mY}{\mathcal{Y}}
\nc{\mZ}{\mathcal{Z}}

\nc \wY {\widetilde{Y}}
\nc \wA {\widetilde{A}}
\nc \tA {\widetilde{A}}
\nc \wR {\widetilde{R}}
\nc \wM {\widetilde{M}}
\nc \tnu {\check{\nu}}
\nc \hnu {\hat{\nu}}
\dmo \Sparse {Sparse}
\dmo \Flat {Flat}
\dmo \Proj	{Proj}
\dmo \circular {circ}
\nc{\muc}{\mu_{\circular}}
\dmo{\Her}{\mathbf{H}}
\dmo{\Res}{\mathbf{R}}
\nc{\tS}{\widetilde{S}}
\nc{\tT}{\widetilde{T}}
\dmo{\dist}{dist}
\dmo{\Switch}{Switch}
\dmo{\bxi}{\bs{\xi}}
\dmo{\eye}{\mathbf{I}_2}
\dmo{\jay}{\mathbf{J}_2}
\dmo{\codeg}{codeg}
\dmo{\discrep}{discrep}
\dmo{\edge}{disc}
\dmo{\expand}{exp}

\dmo{\II}{\mathbb{I}}
\dmo{\indic}{\1}	
\nc{\sph}{\mathbb{S}^{n-1}}
\nc{\ball}{\mathbb{B}^n}
\dmo{\Id}{Id}
\nc{\oneperp}{{\langle \1 \rangle^{\perp}}}
\nc{\iid}{i.i.d.}
\nc{\me}{{m}}
\nc{\mo}{{m}}
\nc{\ka}{k}

\dmo{\Else}{Else}
\nc{\ta}{\boldsymbol{\tau}}
\dmo{\Circ}{{circ}}
\dmo{\wnu}{\widetilde{\nu}}
\nc{\wH}{\widetilde{H}}
\dmo{\schur}{\circ}
\dmo{\co}{co}
\nc{\sig}{\sigma}
\nc{\ha}{\sigma_0}
\dmo{\Comp}{Comp}
\dmo{\Incomp}{Incomp}
\nc{\be}{z}
\nc{\M}{M}
\nc{\kapa}{\kappa}
\nc{\Sig}{\Sigma}

\begin{document}

\title[The circular law for weighted random regular digraphs]{The circular law for random regular digraphs with random edge weights}

\author{Nicholas Cook}
\address{Department of Mathematics, UCLA, Los Angeles, CA, USA 90095-1555}
\email{nickcook@math.ucla.edu}
\thanks{During the completion of this work the author was partially supported by NSF grant DMS-1266164.}

\date{\today}

\begin{abstract}
We consider random $n\times n$ matrices of the form $Y_n=\frac1{\sqrt{d}}A_n\circ X_n$, where $A_n$ is the adjacency matrix of a uniform random $d$-regular directed graph on $n$ vertices, with $d=\lfloor p n\rfloor$ for some fixed $p \in (0,1)$, and $X_n$ is an $n\times n$ matrix of i.i.d.\ centered random variables with unit variance and finite $4+\eta$-th moment (here $\circ$ denotes the matrix Hadamard product). We show that as $n\to \infty$, the empirical spectral distribution of $Y_n$ converges weakly in probability to the normalized Lebesgue measure on the unit disk.
\end{abstract}


\maketitle

\let\oldtocsubsection=\tocsubsection
\renewcommand{\tocsubsection}[2]{\hspace*{1cm}\oldtocsubsection{#1}{#2}}
\let\oldtocsubsubsection=\tocsubsubsection
\renewcommand{\tocsubsubsection}[2]{\hspace*{2.5cm}\oldtocsubsubsection{#1}{#2}}

\setcounter{tocdepth}{2}

\section{Introduction}			\label{sec:intro_ch5}

\subsection{Background}	\label{sec:background_ch5}

For an $n\times n$ matrix $M$ with complex entries, denote by 
\begin{equation}
\mu_M = \frac{1}{n} \sum_{i=1}^n \delta_{\lambda_i(M)}
\end{equation}
its empirical spectral distribution (ESD), where $\lambda_1(M),\dots, \lambda_n(M)\in \C$ are the eigenvalues of $M$, counted with multiplicity, and labeled in some arbitrary fashion. 
When $M$ is a random matrix, $\mu_M$ is a random probability measure.
In the following we denote the space of bounded continuous (resp.\ continuous and compactly supported) real-valued test functions over a space $\mX$ by $C_b(\mX)$ (resp.\ $C_c(\mX)$).

\begin{definition}[Convergence of random measures]		\label{def:weakprob}
Let $(\mu_n)_{n\ge 1}$ be a sequence of random Borel probability measures supported on $\mX= \R$ or $\C$, and let $\mu$ be another Borel probability measure on $\mX$. 
We say $\mu_n$ converges to $\mu$ \emph{weakly in probability} if for all $f\in C_b(\mX)$ the random variable $\int_\mX f d\mu_n$ converges in probability, i.e.\ for any $\eps>0$, 
\begin{equation}	\label{weakprob}
\pro{ \left|\int_{\mX} f d\mu_n - \int_{\mX} f d\mu\right|>\eps} \rightarrow 0 \quad\quad \text{as $n\rightarrow \infty$.}
\end{equation}
We say $\mu_n$ converges to $\mu$ \emph{weakly almost surely} if for all $f\in C_b(\mX)$ the random variable $\int_\mX f d\mu_n$ converges almost surely to $\int_\mX fd\mu$. We similarly define convergence of $\mu_n$ to $\mu$ \emph{vaguely in probability} and \emph{vaguely almost surely} by replacing the space $C_b(\mX)$ with $C_c(\mX)$ above.
\end{definition}

For random matrices with i.i.d.\ entries having finite second moment, the global distribution of the spectrum is asymptotically governed by the \emph{circular law}:
\begin{theorem}[Circular law]		\label{thm:circ}
Let $\xi$ be a complex-valued random variable with mean zero and variance one, and for each $n$ let $X_n=(\xi_{ij})_{1\le i,j\le n}$ be a matrix whose entries are i.i.d.\ copies of $\xi$.
Then the rescaled ESD $\mu_{\frac{1}{\sqrt{n}}X_n}$ converges weakly almost surely to $\mu_{\Circ}$, the uniform measure on $B_{\C}(0,1)$.
\end{theorem}

In the above form \Cref{thm:circ} was established by Tao and Vu in \cite{TaVu:esd}, building on important advances of Girko \cite{Girko84} and Bai \cite{Bai97}.
The circular law was earlier established by Mehta for the case that $\xi$ is a standard complex Gaussian \cite{Mehta}, and later by Edelman in the real Gaussian case \cite{Edelman:circ}. Prior to the work \cite{TaVu:esd} there were versions of \Cref{thm:circ} that made more restrictive assumptions on the distribution of $\xi$ \cite{Bai97, BaSi10:book,GoTi:circ,PaZh,TaVu:circ}.

The circular law has been applied to the study of dynamics of networks arising in fields ranging from neuroscience \cite{SCS:chaos, RaAb:neural, ARS:block} to ecology \cite{May72, AGBTAM15:foodwebs}.
\Cref{thm:circ} is mathematically interesting as an instance of the \emph{universality phenomenon} in random matrix theory: in the limit as $n\rightarrow \infty$, the global distribution of the spectrum of $X_n$ is determined completely by the first two moments of the entries $\xi_{ij}$. 
The theorem hence identifies a wide class of matrix ensembles lying in the \emph{circular law universality class}. 
For additional background and history on the circular law we point the reader to the survey \cite{BoCh:survey}. 

It is natural to ask whether this universality class extends to include some ensembles not covered by \Cref{thm:circ}, that is, if we can relax the independence, identical distribution, or moment hypotheses on the entries. 
Informally, one might expect an ensemble $M_n=(\xi_{ij})$ to exhibit the circular law if the entries $\xi_{ij}$ are centered around a common value, have reasonably bounded tails, are not too degenerate (i.e.\ their distributions do not concentrate too much near deterministic values), and are only weakly correlated. 

For i.i.d.\ matrices, the second moment hypothesis is sharp,
and the limiting spectral distribution for certain classes of heavy tail matrices was described in \cite{BCC:heavy}. 
Nevertheless, in \cite{Wood:sparse}, Wood showed that the circular law is robust under \emph{sparsification}: letting $X_n=(\xi_{ij})$ be an i.i.d.\ matrix and $B_n=(b_{ij})$ a matrix of i.i.d.\ Bernoulli($p$) indicator variables, independent of $X_n$, the empirical spectral distribution of $M_n:=\frac{1}{\sqrt{np}} B_n\schur X_n$ converges weakly in probability (see \Cref{def:weakprob} below) to the circular law, provided that $p=p(n)\gg n^{\eps -1}$ for any fixed $\eps>0$. 
(Here $\schur$ denotes the matrix Hadamard product, i.e.\ $B_n\schur X_n = (b_{ij}\xi_{ij})$.)
The lower bound on $p$ is near-optimal, as it is easy to see that for $p\asymp 1/n$ the matrix $M_n$ will have a linear proportion of trivial columns, so that the limiting spectral distribution will have an atom at zero. 

There has been considerable interest in extending the circular law to cover models with dependent entries. 
In \cite{Nguyen:uds}, Nguyen showed that the circular law holds for matrices drawn uniformly (according to volume measure) from the Birkhoff polytope of doubly stochastic matrices. 
This followed work of Bordenave, Caputo and Chafa\"i on random stochastic matrices, which still enjoy joint independence of the matrix rows. 
In \cite{AdCh}, Adamczak and Chafa\"i established the circular law for matrices whose distribution over $\mM_n(\R)$ is a log-concave and unconditional measure, generalizing Edelman's result for Gaussian matrices \cite{Edelman:circ}. 
Together with Wolff, the same authors showed in \cite{ACW:exchangeable} that the circular law holds for a matrix whose entries are exchangeable and have a finite moments of order $20+\eps$. 

A particularly motivating challenge in random matrix theory is to extend universality results to include adjacency matrices of random regular graph models. 
For integers $n\ge 1$ and $d\in [n]$ denote
\begin{equation}	\label{def:And}
\mA_{n,d} = \left\{ A \in \{0,1\}^{n\times n} :\; A\1 = A^\tran\1 = d\1 \right\},
\end{equation}
which is the set of 0--1 adjacency matrices for $d$-regular directed graphs (digraphs) on $n$ vertices, allowing self-loops.
(Here and throughout, $\1=\1_n$ denotes the column vector of all 1s.) 
One can view 
a uniform random element $A\in \mA_{n,d}$ 
as a discrete analogue of the doubly stochastic matrix considered by Nguyen in \cite{Nguyen:uds}. We note also that such matrices are not covered by the above-mentioned results in \cite{ACW:exchangeable}: while the rows and columns of $A$ are separately exchangeable, the entries themselves are not.

More generally one can consider random regular digraphs with random edge weights, i.e.\ matrices of the form $A\circ X$ with $A\in\mA_{n,d}$ uniform random and $X$ an i.i.d.\ matrix independent from $A$. For applications in neuroscience, where random matrices have been used as mathematically tractable models of synaptic matrices \cite{SCS:chaos,RaAb:neural, ARS:block}, random matrices supported on adjacency matrices for directed graphs are of interest  as they reflect the sparsity of natural neural networks.

We have the following augmented version of a conjecture of Bordenave and Chafa\"i from \cite{BoCh:survey} (only case (2) below is stated there).
\begin{conjecture}[Bordenave--Chafa\"i \cite{BoCh:survey}]		\label{conj:rrd}
Let $1\le d\le n-1$ and let $A$ be a uniform random element of $\mA_{n,d}$.
\begin{enumerate}
\item If $d=d(n)\le n/2$ satisfies $d\rightarrow \infty$ as $n\rightarrow \infty$, then $\mu_{\frac{1}{\sqrt{d}}A}$ converges to the circular law as $n\to \infty$. 
\item  If $d\ge3$ is fixed independent of $n$, then as $n\to\infty$, $\mu_{A}$ converges to the \emph{oriented Kesten--McKay law} on $\C$, with density given by
\begin{equation}	\label{kmlaw}
f_{KM}(w)= \frac1\pi\frac{d^2(d-1)}{(d^2-|w|^2)^2}1_{\set{|w|\le \sqrt{d}}}.
\end{equation}

\end{enumerate}

\end{conjecture}
From the above one obtains results for $n/2<d\le n-1$ by considering the complementary matrix with entries $1-a_{ij}$, which has the same limiting spectral distribution.
See \Cref{fig:1} for some numerical evidence supporting \Cref{conj:rrd}.

The explicit density in \eqref{kmlaw} can be computed as the \emph{Brown measure} of the \emph{free sum} of $d$ Haar unitary operators -- see \cite[Example 5.5]{HaLa}. 
We note that in \eqref{kmlaw}, if one rescales $w$ by $\sqrt{d}$ and sends $d\rightarrow \infty$ the expression converges to the normalized indicator for the unit disk, which gives some evidence for part (1) of the conjecture. 
By contiguity results (see for instance \cite{Janson:contiguity}), to establish Conjecture \ref{conj:rrd} for fixed $d$ it suffices to consider a different measure, the sum of $d$ i.i.d.\ uniform permutation matrices. 
It was shown by Basak and Dembo in \cite{BaDe} that the sum of $d$ i.i.d.\ Haar unitary or orthogonal matrices has limiting spectral distribution given by \eqref{kmlaw}, so Conjecture \ref{conj:rrd} posits that their result should hold if the unitaries are restricted to permutation matrices. 

\begin{figure}
\includegraphics[width=350pt]{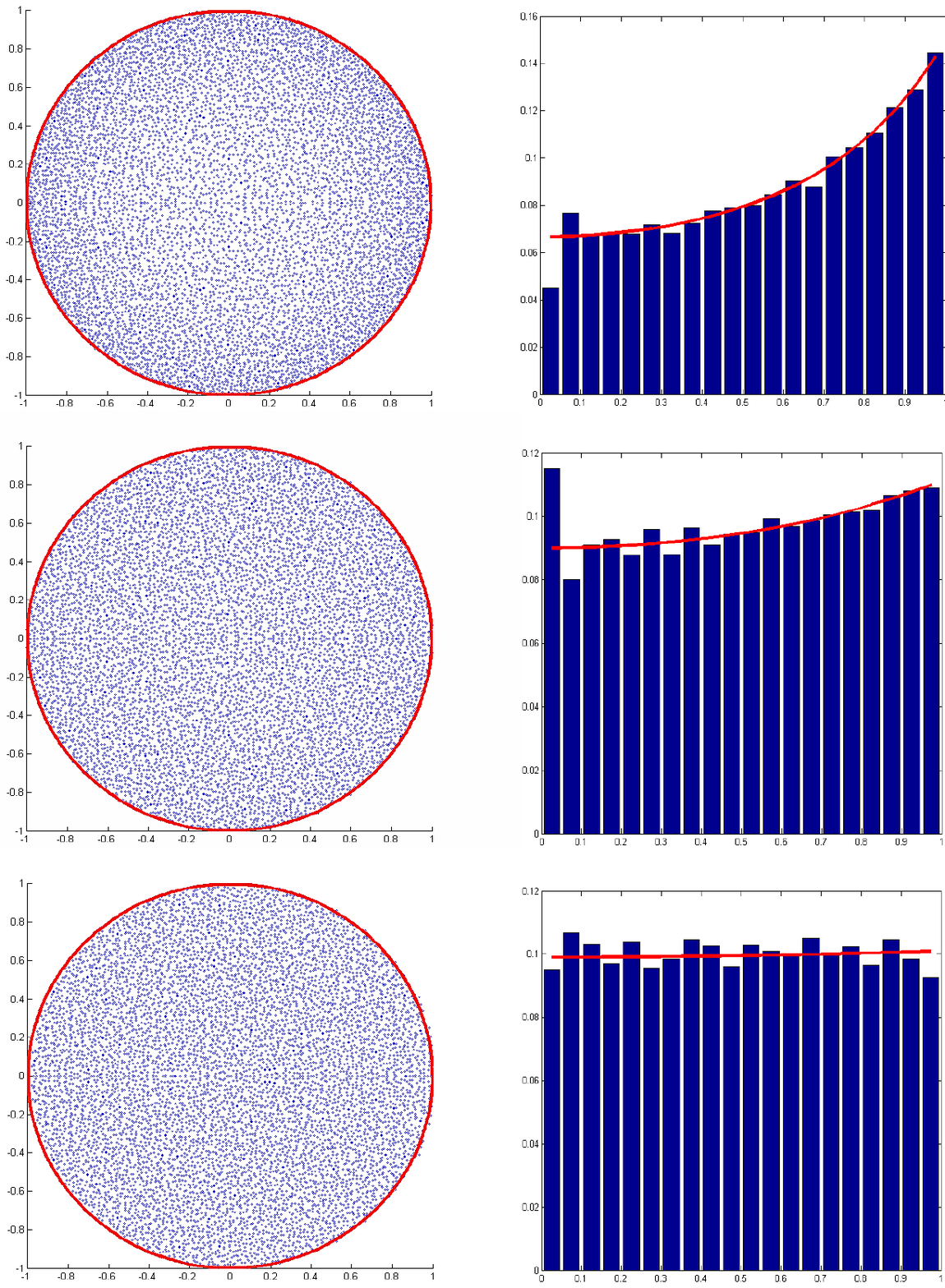}
\caption[Empirical eigenvalue distributions for simulated 
adjacency matrices]{Empirical eigenvalue distributions for simulated $8000\times 8000$ 
uniform random elements of $\mA_{n,d}$, rescaled by $\sqrt{d}$, 
for $d= 3$ (top), 10 (middle), and 100 (bottom). 
Left: scatterplots of eigenvalues, with the unit circle plotted in red for reference. 
Right: histograms for eigenvalue moduli, with each bin count normalized by $2\pi$ times the distance to the origin. The curves predicted by \eqref{kmlaw} are plotted in red.
While the eigenvalue distribution is noticeably more dense near the edge of the support for $d=3$, for $d=100$ it is indistinguishable from the uniform distribution on the disk.} \label{fig:1}

\end{figure}

The two cases of Conjecture \ref{conj:rrd} parallel known results for undirected regular graphs. 
Namely, it was shown by McKay in \cite{McKay:esd} that for $d\ge 3$ fixed, the limiting spectral distribution of adjacency matrices 
of $d$-regular undirected graphs is given by the (explicit) Kesten--McKay law.
More recently, it has been shown in \cite{DuPa}, \cite{TVW} and \cite{BKY} that if $d\rightarrow \infty$ at certain speeds, the empirical spectral distribution 
converges to the \emph{semicircular law}, matching the asymptotic behavior of Wigner matrices (the Hermitian analogue of i.i.d.\ matrices). 
Moreover, these results show that the convergence of ESDs holds on \emph{mesoscopic scales}, i.e.\ on intervals with length shrinking to zero relative to the limiting support of the spectrum, but growing to infinity relative to the mean spacing of eigenvalues. 
In particular, the work \cite{BKY} has shown convergence at the optimal mesoscopic scale, provided $d$ grows in the range $f(n)\ll d\ll \big(\frac{n}{f(n)}\big)^{2/3}$ for some $f(n)$ growing poly-logarithmically. Very recently, the Kesten--McKay law has been established at the optimal mesoscopic scale for sufficiently large fixed $d$ \cite{BHY:km}.


\subsection{Main result}	\label{sec:main_ch5}

In the present work we obtain some partial progress towards Conjecture \ref{conj:rrd} by considering random regular digraphs with random edge weights, i.e.\ matrices of the form
$$Y= A\schur X$$
where $X=(\xi_{ij})$ 
is a matrix with i.i.d.\ centered entries of unit variance. 
One can view this as 
a random regular digraph adjacency matrix with some additional randomness, or alternatively as an i.i.d.\ matrix that has been randomly sparsified to have each row and column supported on $d$ entries. This is to be compared with the work of Wood \cite{Wood:sparse} discussed above, where the sparsification is performed by i.i.d.\ Bernoulli indicators.

\begin{theorem}[Main result]		\label{thm:main}
Fix $p\in (0,1)$, and let $\xi\in\C$ be a centered random variable with unit variance and $\e|\xi|^{4+\eta}<\infty$ for some fixed $\eta>0$.
For each $n$, let $A_n$ be a uniform random element of 
$\mA_{n,d}$,
where we write $d=\lf p n\rf$, and let $X_n=(\xi_{ij})_{1\le i,j\le n}$ be a matrix of i.i.d.\ copies of $\xi$, independent of $A_n$.
Denoting $Y_n=A_n\schur X_n$, we have that the rescaled ESDs $\mu_{\frac{1}{\sqrt{np}}Y_n}$ converge weakly in probability to $\mu_{\Circ}$, the uniform measure on $B_{\C}(0,1)$.

Moreover, if we assume $\xi$ is a standard real Gaussian variable then the rescaled ESDs converge weakly almost surely. 
\end{theorem}


\begin{remark}[Moment assumption on $\xi$]		\label{rmk:mom}
The assumption $\e |\xi|^{4+\eta}<\infty$ is used to ensure that $A_n\circ X_n$ has operator norm of size $O(\sqrt{n})$ with high probability, which is needed in our arguments to obtain control on small singular values. However, we believe it is likely this hypothesis can be eliminated. 
\end{remark}

In a sequel to this paper \cite{Cook:rrdcirc} we establish the circular law for the unsigned adjacency matrix $A_n$ under the mild growth condition $\min(d,n-d)\ge \log^Cn$ for a suitable constant $C>0$, which establishes part (1) of \Cref{conj:rrd} for this range of $d$. That result requires significant new ideas over the proof of \Cref{thm:main} due both to the sparsity of $A_n$ and the lack of the extra randomness of $X_n$. The work \cite{Cook:rrdcirc} builds on some of the basic arguments of the present paper, which can serve as a warmup for the reader interested in understanding the proofs. Below, we provide some comparison between \cite{Cook:rrdcirc} and the present work after first giving some background on the general strategy. We also mention that in the recent work \cite{BCZ:perm} we have proved the circular law for the sum of $d$ i.i.d.\ uniform random $n\times n$ permutation matrices (under a growth condition on $d=d(n)$) which provides a model for random $d$-regular directed multigraphs.

Let us denote $\mu_n:=\mu_{\frac{1}{\sqrt{np}}Y_n}$.
Our general approach to the proof of \Cref{thm:main} follows previous works in using Girko's Hermitization strategy, which reduces the problem of establishing convergence of linear statistics $\int_\C f d\mu_n$ to one of establishing convergence of the \emph{logarithmic potentials}
$
\int_{\R^+} \log(s) d\nu_{n,z}
$
for a.e.\ $z\in \C$, where
\begin{equation}	\label{nunz}
\nu_{n,z}:=\mu_{\sqrt{(\frac{1}{\sqrt{np}}Y_n-z)^*(\frac{1}{\sqrt{np}}Y_n-z)}}
\end{equation}
is the empirical singular value distribution of the scalar-shifted matrix $\frac{1}{\sqrt{np}}Y_n-z$.
Following the influential work of Bai \cite{Bai97}, this in turn can be reduced to the following three tasks:
\begin{enumerate}
\item (Weak convergence of singular value distributions). Show that for a.e.\ $z\in \C$, the measures $\nu_{n,z}$ converge weakly in probability.	\label{step1}
\item (Smallest singular value). Show that for a.e.\ $z\in \C$, with probability $1-o(1)$ the smallest singular value of $\frac{1}{\sqrt{np}}Y_n-z$ is of order at least $e^{-o(n)}$.	\label{step2}
\item (Wegner-type estimate). Show that for a.e.\ $z\in \C$ and some $a, c>0$ depending on $z$ (but not on $n$), with probability $1-o(1)$, for all $\eta\in (n^{-c},a]$ we have $\nu_{n,z}([0,\eta]) = O(\eta)$.	\label{step3}
\end{enumerate}
In the above, the implied constants and rates of convergence may all depend on the choice of $z$.

A key difficulty for 
all of these problems is the lack of independence among the entries of $A_n$, and to deal with this we combine several results from the literature, as well as a new Wegner-type estimate.
For 
(\ref{step1}) we use 
a conditioning argument of Tran, Vu and Wang \cite{TVW} to compare with a matrix $\widetilde{Y}_n=B_n\circ X_n$, where $B_n$ is an i.i.d.\ Bernoulli matrix of the same expected density as $A_n$. Here we apply an asymptotic for the number of 0--1 matrices with constrained row and column sums due to Canfield and McKay \cite{CaMc}.
For 
(\ref{step2}) and (\ref{step3}) 
we leverage the independence of the entries of $X_n$ together with robust connectivity properties of random regular digraphs.
Specifically, we use a graph discrepancy result of the author from \cite{Cook:discrep} to prove the 
matrix $A_n$ satisfies a certain ``broad connectivity" condition with high probability.
This gives us access to recent results on the invertibility properties of random matrices with independent but non-identically distributed entries having broadly-connected support, due to Rudelson--Zeitouni \cite{RuZe} (for the Gaussian case) and the author \cite{Cook:ssv} (for the general case). 
In the process we prove \Cref{thm:broad_small}, which gives a Wegner-type bound at essentially optimal scale 
(i.e.\ allowing one to take any $c\in (0,1)$ in (\ref{step3}))
for the small singular values of random matrices with broadly-connected support.

In the sequel \cite{Cook:rrdcirc} to this paper, the proof of the circular law for the unsigned, possibly sparse adjacency matrices $A_n$ follows the same basic steps as in (\ref{step1})--(\ref{step3}) above, but each of these tasks is significantly more involved.
In this setting we are no longer able to apply the recent results of \cite{RuZe,Cook:ssv} for (\ref{step2}) and (\ref{step3}) as these only apply to matrices of the form $A_n\circ X_n$ with $X_n$ an i.i.d.\ matrix and $A_n$ a matrix independent of $X_n$ and having non-vanishing density. The bulk of \cite{Cook:rrdcirc} is spent establishing (\ref{step2}), building on previous works on the invertibility problem \cite{Cook:sing,LLTTY}, and making use of finer graph expansion properties than the above-mentioned broad connectivity property. For (\ref{step1}) and (\ref{step3}) we apply the conditioning argument of \cite{TVW} as in the present work to replace the matrix $A_n$ with an i.i.d.\ Bernoulli matrix of the same density. Unlike in the present work, we then need an extra step of replacing the (possibly sparse) Bernoulli matrix with a (dense) Gaussian matrix (using a Lindeberg swapping argument) at which point standard arguments apply. See \cite{Cook:rrdcirc} for further details.

\subsection{Outline} 	
The rest of the paper is organized as follows.
In \Cref{sec:highlevel_ch5} we give a high level proof of \Cref{thm:main} following the Hermitization approach, reducing our task to establishing the weak convergence of empirical singular value distributions \eqref{nunz}, and lower bounds on small singular values of $\frac1{\sqrt{np}}Y_n-z$.
The weak convergence of $\nu_{n,z}$ is proved in \Cref{sec:svconv}.
In \Cref{sec:redux}, we present two results -- \Cref{prop:broad} and \Cref{thm:broad_small} -- on the small singular values of random matrices with broadly-connected support, and prove that random regular digraphs satisfy this condition with high probability.
In \Cref{sec:wegner} we prove \Cref{thm:broad_small}.

\subsection{Notation}	\label{sec:notation}

We use standard asymptotic notation, always with respect to the limit $n\rightarrow\infty$. 
$f=O(g)$, $f\ll g$ and $g\gg f$ are all synonymous to the statement that $|f|\le Cg$ for some absolute constant $C>0$. 
$f=o(g)$ 
means $f/g\rightarrow 0$ as $n\rightarrow \infty$.
Dependence of implied constants on parameters is indicated with subscripts. 
$C,C', c, c_0$, etc.\ denote constants which may change from line to line. 

For $J\subset [n]$, we denote by $\C^J\subset\C^n$ (resp. $S^J\subset S^{n-1}$) the set of vectors (resp. unit vectors) in $\C^n$ supported on $J$.
Given a vector $v\in \C^n$, we denote by $v_J\in \C^n$ the projection of $v$ to the coordinate subspace $\C^J$.
${[m]\choose x}$ denotes the family of subsets of $[m]$ of size $\lf x\rf$. We write $E_\rho$ for the closed Euclidean $\rho$-neighborhood of a set $E\subset \C^n$.

$\mM_{n,m}(\C)$ denotes the set of $n\times m$ matrices with complex entries, and for square matrices we simply write $\mM_n(\C)= \mM_{n,n}(\C)$. 
For a matrix $A=(a_{ij})\in \mM_{n,m}(\C)$ we will sometimes use the notation $A(i,j)= a_{ij}$.
For $I\subset[n],J\subset[m]$ with elements $i_1<\cdots<i_{|I|}$ and $j_1<\cdots j_{|J|}$, respectively, $A_{I,J}$ denotes the $|I|\times |J|$ submatrix $(a_{i_k,j_l})_{1\le k\le |I|, 1\le l\le |J|}$.
We often abbreviate $A_{J}:= A_{J,J}$.
$\|\cdot \|$ denotes the $\ell^2_m\rightarrow\ell^2_n$ operator norm when applied to $n\times m$ matrices, and the Euclidean norm when applied to vectors.
Other norms are indicated with subscripts. 

For a measure $\mu$ on a space $\mX$ and a function $f:\mX\to \C$, we will sometimes abbreviate $\mu(f):= \int_{\mX} fd \mu$.

When there can be no confusion we will suppress the subscript $n$ from the matrices $X_n, A_n, Y_n$ etc.


\section{Reduction to estimates on singular values}	\label{sec:highlevel_ch5}

We follow Girko's \emph{Hermitization method}, introduced in \cite{Girko84}, and which we briefly review below; see \cite{TaVu:esd} or the survey \cite{BoCh:survey} for more background.
For a probability measure $\mu$ on $\C$ we define the \emph{log-potential}
\begin{equation}	\label{deflogpot}
U_\mu(z) :=\int_{\C}\log|w-z|d\mu(w).
\end{equation}
Let us abbreviate $\mu_n:=\mu_{\frac{1}{\sqrt{np}}Y_n}$. 
The following is taken from \cite[Theorem 1.20]{TaVu:esd}:

\begin{lemma}[Hermitization]		\label{lem:herm}
Let $\mu$ be a probability measure on $\C$ satisfying $\int_{\C} |z|^2d\mu(z)<\infty$.
The following are equivalent:
\begin{enumerate}[(i)]
\item $\mu_n$ converges weakly in probability to $\mu$.
\item For almost every $z\in \C$, $U_{\mu_n}(z)$ converges in probability to $U_\mu(z)$.
\end{enumerate}
Moreover, the same equivalence holds with ``in probability" replaced by ``almost surely". 
\end{lemma}

Denoting the ordered singular values of a matrix $M\in \mM_n(\C)$ by $s_1(M)\ge\cdots\ge s_n(M)$, from the well-known identity 
$$\prod_{i=1}^n|\lambda_i(M)| = |\det(M)| = \prod_{i=1}^n s_i(M)$$
we have
\begin{align*}
U_{\mu_n}(z) &= \frac{1}{n}\log\left|\det\left(\frac{1}{\sqrt{np}}Y-zI\right)\right| \\
&= \frac{1}{n} \sum_{i=1}^n \log s_i\left(\frac{1}{\sqrt{np}} Y-zI\right)\\
&=\int_{\R_+} \log(s) d\nu_{n,z}(s)
\end{align*}
where we have defined
\begin{equation}
\nu_{n,z}:= \frac1n\sum_{i=1}^n \delta_{s_i\left(\frac{1}{\sqrt{np}}Y-zI\right)}.
\end{equation}
The log-potential for the uniform measure on the unit disk is 
\begin{equation}	\label{circlogpot}
U(z):=\frac{1}{\pi}\int_{B_{\C}(0,1)} \log|w-z|dw= 
\begin{cases}
\;\log |z| & \text{if } |z|>1\\
-\frac12(1-|z|^2) & \text{otherwise}.
\end{cases}
\end{equation}
From the above lines and \Cref{lem:herm}, to establish \Cref{thm:main} it suffices to show that for almost every $z\in \C$, $\int_{\R_+}\log(s)d\nu_{n,z}(s)$ converges in probability to $U(z)$.

As a first step we will establish the following:
\begin{proposition}[Weak convergence of $\nu_{n,z}$]	\label{prop:sing}
For all $z\in \C$, $\nu_{n,z}$ converges vaguely almost surely to a deterministic probability measure $\nu_z$ on $\R_+$.
Furthermore, for all $z\in \C$ the measures $\nu_z$ satisfy 
\begin{equation}	\label{nuzlogp}
\int_{\R_+}\log(s)d\nu_z(s) = U(z)
\end{equation}
where $U(z)$ was defined in \eqref{circlogpot}.
\end{proposition}


In order to deduce from this the convergence of log-potentials $U_{\mu_n}(z)$, we will apply the above proposition to a truncation $f\in C_c(\R_+)$ of the function $s\mapsto \log(s)$. 
To show that the truncated integral $\int_{\R_+} fd\nu_{n,z}$ is a good approximation of $U_{\mu_n}(z)$, we must prove that the measures $\nu_{n,z}$ uniformly integrate the singularities of $s\mapsto \log(s)$.

The singularity at $+\infty$ requires control on a matrix norm of $\frac{1}{\sqrt{np}}Y-zI$. While the standard approach is to use the Hilbert--Schmidt norm, we will instead apply the following stronger control on the operator norm as we will need it later. 

\begin{lemma}[The largest singular value]	\label{lem:opnorm}
Except with probability $O(n^{-\eta/8})$ we have $s_1(\frac{1}{\sqrt{np}}Y-zI)=O_{z,p}(1)$.
If $\xi$ is a standard Gaussian then the probability bound improves to $O(e^{-cn})$ for some constant $c>0$.
\end{lemma}


\begin{proof}
This is a consequence of a general bound on the expected operator norm of random matrices due to Lata\l a \cite{Latala} and a standard truncation argument. See the proof of \cite[Lemma 5.5]{Cook:ssv} for the details. The bound for the Gaussian case is well-known; see for instance \cite[Lemma 3.1]{RuZe}.
\end{proof}


The singularity of the logarithm at zero requires lower bounds on the small singular values of $\frac{1}{\sqrt{np}}Y-zI$. 
It turns out that for $z$ inside the support of the limiting ESD (i.e.\ $|z|\le 1-\eps$), the limiting singular value distributions $\nu_z$ from \Cref{prop:sing} have a positive density in a neighborhood of zero.
Based on this fact (which is not needed for the proof), we might expect that the smallest singular value $s_n(\frac1{\sqrt{np}}Y-zI)$ to live at scale $\sim 1/n$, and in general that the $k$-th smallest singular value would be at scale $\sim k/n$. 
The following two propositions provide quantitative lower bounds at these scales.

\begin{proposition}[The smallest singular value]		\label{prop:least}
For any $t>0$, 
\begin{equation}	\label{ssv:general}
\pro{ s_n(Y-z\sqrt{np}I) \le tn^{-1/2}} \ll_{z,p} t + n^{-1/2}.
\end{equation}
Moreover, if $\xi$ is a standard Gaussian variable then we have that for any $K>0$,
\begin{equation}	\label{ssv:gaussian}
\pro{ s_n(Y-z\sqrt{np}I) \le n^{-K-1/2}} \ll_{z,p,K} n^{-K}.
\end{equation}
\end{proposition}

\begin{proposition}[Local Wegner estimate]	\label{prop:small}
There are constants $a_0,a_1,a_2>0$ depending only on $p,z$ such that for all $\alpha \in (0,1)$, except with probability $O(\exp(-a_0n^\alpha))$, for all $i\in [n^\alpha, a_1n]$ we have
\begin{equation}
s_{n-i}\left(\frac{1}{\sqrt{np}}Y-zI\right)  \ge a_2\frac{i}{n}.
\end{equation}

\end{proposition}

We defer the proofs to Sections \ref{sec:redux} and \ref{sec:wegner}, and conclude the proof of \Cref{thm:main} on Propositions \ref{prop:sing}, \ref{prop:least}, \ref{prop:small} and \Cref{lem:opnorm}.
We will first argue the general case; at the end we will discuss the necessary modifications for the Gaussian case. 


Fix $z\in \C$. 
Here we allow implied constants to depend on $\eta,p$ and $z$.
Our aim is to show that for any $\eps>0$,
\begin{equation}	\label{aimlog}
\pro{\big|\nu_{n,z}(\log) - \nu_z(\log)\big|>\eps} =o(1)
\end{equation}
with $\nu_z$ as in \Cref{prop:sing}.
From \Cref{lem:opnorm} we have 
\begin{equation}	\label{upper0}
\pro{\nu_{n,z}([C_0,\infty)) =0}=1-o(1)
\end{equation}
for some $C_0>0$ sufficiently large depending only on $z$. 
In particular, with \Cref{prop:sing} this implies that $\nu_z$ is supported on $[0,C_0)$.
For $\delta>0$ small, let $f_\delta\in C_c(\R_+)$ satisfy
$$f_\delta(s)= 
\begin{cases}
0 & s\in [0,\delta/2]\cup[2C_0,\infty)\\
\log(s) & s\in [\delta,C_0]
\end{cases}
$$
and take $f_\delta$ to be linearly increasing and decreasing on the intervals $[\delta/2,\delta]$ and $[C_0,2C_0]$, respectively. 
From \Cref{prop:sing} we have that for any fixed $\delta>0$, $\nu_{n,z}(f_\delta) \to \nu_z(f_\delta)$ almost surely,
so from \eqref{upper0} it only remains to show that with high probability,
\begin{equation}	\label{remainseta}
\int_0^\delta |\log(s)| d\nu_{n,z}(s) = O(\delta^{0.9})
\end{equation}
(say) for all $\delta>0$ sufficiently small. 

For ease of notation we write $s_i$ for $s_i(\frac{1}{\sqrt{np}}Y-zI)$. 
Fix $\alpha\in (0,1)$ arbitrarily. 
We write left hand side of \eqref{remainseta} as
\begin{equation}	\label{finish.split}
\frac{1}{n} \sum_{i=0}^{n-1} |\log(s_{n-i})|\un(s_{n-i}\le \delta) = \frac1n\sum_{i=0}^{\lf n^{\alpha}\rf} |\log(s_{n-i})| + \frac1n\sum_{i>n^\alpha} |\log(s_{n-i})|\un(s_{n-i}\le \delta).
\end{equation}
By \eqref{ssv:general} in \Cref{prop:least}, except with probability $O(n^{-1/2})$, the first term on the right hand side is bounded by $O(n^{1-\alpha} \log n) =o(1)$. 
By \Cref{prop:small}, taking $\delta<a_1a_2$ we have except with probability $O(\expo{-a_0n^\alpha})$, the second term is bounded by
$$
\frac{1}{n} \sum_{1\le i<\delta n/a_2} \left|\log\left(a_2\frac{i}{n}\right)\right| \ll \frac{1}{n}\sum_{1\le i<\delta n/a_2} \left(a_2\frac{i}{n}\right)^{-0.1}\ll \int_0^{2\delta/a_2} s^{-0.1}ds = O(\delta^{0.9})
$$
which completes the proof. 

Now for the Gaussian case, our aim is to show $\nu_{n,z}(\log)\to \nu_z(\log)$ almost surely. 
By the Borel--Cantelli lemma it suffices to establish \eqref{aimlog} with a bound that is summable in $n$ for any fixed $\eps>0$. 
Fixing $\eps>0$, we only need to verify that the $o(1)$ term in \eqref{upper0} is summable in $n$, and that \eqref{remainseta} holds with probability $1-O(n^{-2})$, say.
For the former, by \Cref{lem:opnorm} we can replace the bound $o(1)$ in \eqref{upper0} with $O(e^{-cn})$. 
For the latter we only need to replace the application of \eqref{ssv:general} with \eqref{ssv:gaussian}, taking $K=2$. 
This completes the proof for the Gaussian case of \Cref{thm:main}.

\section{Weak convergence of singular value distributions}		\label{sec:svconv}

In this section we prove \Cref{prop:sing}.
Specifically, we show that for any $f\in C_c(\R_+)$ and any $\eps>0$, 
\begin{equation}	\label{goal0_ch5}
\pro{|\nu_{n,z}(f) -\nu_z(f)|>\eps} \ll_p \expo{ -c_0pn^2}
\end{equation}
for some $c_0=c_0(\eps,f)>0$, where here and in the remainder of the section we allow implied constants to depend on $f$ and $\eps$. 
The desired result will then follow from the Borel--Cantelli lemma.

Following an approach of Tran, Vu and Wang from \cite{TVW}, we compare the signed 
adjacency matrix $Y$ with an i.i.d.\ matrix taking the form
$$\widetilde{Y} =  (b_{ij}\xi_{ij})_{1\le i,j\le n} = B\schur X$$
where the variables $b_{ij}$ are Bernoulli($p$) indicator variables (recall that $d=\lf p n\rf$), jointly independent of each other and of the matrix $X$. 
Note that the entries of $\frac{1}{\sqrt{p}}\wY$ are i.i.d.\ and centered with unit variance.
Writing $\wnu_{n,z}:=\frac{1}{n}\sum_{i=1}^n \delta_{s_i(\frac{1}{\sqrt{np}}\wY-zI)}$, 
it follows from a result of Dozier and Silverstein \cite{DoSi07a} that for all $z\in \C$ the measures $\wnu_{n,z}$ converge weakly almost surely to a deterministic probability measure $\nu_z$;  see \cite[Lemma 3]{PaZh} for the verification that \eqref{nuzlogp} holds with $\mu=\muc$.
In particular, we have that for any $f\in C_b(\R_+)$, 
\begin{equation}	\label{econverge}
|\nu_z(f) -\e \wnu_{n,z}(f)| = o(1)
\end{equation}
so it suffices to show
\begin{equation}		\label{goal1}
\pro{|\nu_{n,z}(f) - \e\wnu_{n,z}(f)|>\eps} \ll \expo{ -c_0pn^2}.
\end{equation}

Define the event
\begin{equation}
\event_{n,d}:= \big\{ B\in \mA_{n,d}\big\} 
\end{equation}
that the Bernoulli matrix $B$ is the adjacency matrix of a $d$-regular digraph. 
Note that $B|\event_{n,d}\eqd A.$
Hence, to bound an event $\set{A\in \mA_0}$ for some set $\mA_0\subset \mA_{n,d}$, we can bound the corresponding event $\set{B\in \mA_0}$ and apply the bound
\begin{equation}	\label{condid}
\pro{A\in \mA_0} = \frac{\pro{B\in \mA_0\cap\mA_{n,d}}}{\pro{\event_{n,d}}} \le \frac{\pro{B\in \mA_0}}{\pro{\event_{n,d}}}
\end{equation}
together with a lower bound on $\pro{\event_{n,d}}$. 
For this we have the following result of Canfield and McKay from \cite{CaMc}.

\begin{lemma}	\label{lem:tran}
Assume $\min(d,n-d)\gg n/\log n$. Then
\begin{equation}
\pro{ \event_{n,d}} = (1+o(1))\sqrt{2\pi d(n-d)} \expo{ -n\log\left(\frac{2\pi d(n-d)}{n}\right) }.
\end{equation}
In particular, if $d=\lf pn\rf$ for some fixed $p\in (0,1)$, then
\begin{equation}
\pro{ \event_{n,d}} \ge \expo{ -O_p(n\log n)}.
\end{equation}

\end{lemma}

By the above lemma, \eqref{condid} and \eqref{goal1}, it suffices to show that for any $f\in C_c(\R_+)$ and any $\eps>0$,
\begin{equation}		\label{goal2}
\pro{|\wnu_{n,z}(f) - \e\wnu_{n,z}(f)|>\eps} \ll \expo{ -c_0pn^2}
\end{equation}
for some $c_0(\eps,f)>0$.

We have hence reduced our task to proving a concentration bound for linear statistics of the singular value distribution of the perturbed i.i.d.\ matrix $\frac{1}{\sqrt{np}}\wY - zI$. 
For this task we have the following lemma, which follows from the work of Guionnet and Zeitouni in \cite{GuZe}:
\begin{lemma}[Concentration of linear statistics]	\label{lem:guze}
Let $H=(h_{ij})_{1\le i,j\le n}$ be a Hermitian random matrix, and assume the variables on and above the diagonal are jointly independent and that $|h_{ij}|\le K/\sqrt{n}$ uniformly in $i,j$ for some $K\in (0,\infty)$. 
Let $f:\R\rightarrow \R$ be an $L$-Lipschitz function supported on a compact interval $I\subset\R$, and let $H_0$ be an arbitrary deterministic $n\times n$ Hermitian matrix. 
Then for any $\delta>0$,
\begin{equation}
\pr\big( \left| \mu_{H+H_0}(f)- \e \mu_{H+H_0}(f)\right| \ge \delta\big) \le \frac{C|I|}{\delta}\expo{ -\frac{cn^2\delta^4}{K^2L^2|I|^2}}
\end{equation}
for some absolute constants $C,c>0$.
\end{lemma}

\begin{proof}
For the case that $H_0=0$, this follows directly from \cite[Theorem 1.3(a)]{GuZe}. 
For the general case, the only part of the argument that needs modification is in the proof of their Theorem 1.1(a), where we need to show that for $f:\R\rightarrow \R$ convex and $L$-Lipschitz, the function 
$H\mapsto \mu_{H+H_0}(f)$ is a convex and $O(L)$-Lipschitz function on the space of Hermitian matrices.
However, this follows directly from their Lemma 1.2 and the fact that the convexity and Lipschitz properties are invariant under translations $H\mapsto H+H_0$.
The rest of the proofs of \cite[Theorem 1.1(a)]{GuZe} and \cite[Theorem 1.3(a)]{GuZe} apply with no modification.
\end{proof}

Fix $\eps>0$ and $f\in C_c(\R)$.
To apply the above concentration estimate to the measures $\wnu_{n,z}$, we recall the linearization approach to the study of singular value distributions of random matrices. 
From $\wY$ we form a $2n\times 2n$ Hermitian matrix 
\begin{equation}
H(z)= \frac{1}{\sqrt{np}}\begin{pmatrix} 0 & \wY - z\sqrt{np}I\\ (\wY-z\sqrt{np}I)^* & 0\end{pmatrix}.
\end{equation}
It is routine to verify that the $2n$ eigenvalues of $H(z)$, counted with multiplicity, are $\{\pm s_i(\frac{1}{\sqrt{np}}\wY -zI)\}_{i=1}^n$.
In terms of empirical spectral distributions, $\mu_{H(z)}$ is the symmetrization across the origin of the measure $\wnu_{n,z}$ on $\R_+$. 
From \eqref{goal2}, it now suffices to show
\begin{equation}	\label{goal3}
\pro{|\mu_{H(z)}(f) - \e\mu_{H(z)}(f)|>\eps} \ll \expo{ -c_0pn^2}.
\end{equation}

As $f$ is uniformly continuous, there exists $\delta=\delta(\eps,f)>0$ such that $|f(s)-f(t)|\le \eps/10$ whenever $|s-t|\le \delta$. By linear interpolation on a $\delta$-mesh of the support of $f$ we may find a function $f_\eps\in C_c(\R)$ with Lipschitz constant $O(\eps/\delta)=O_{\eps,f}(1)$, and such that $\|f-f_\eps\|_{\infty}\le \eps/10$. 
It now suffices to show 
\begin{equation}	\label{goal4}
\pro{|\mu_{H(z)}(f_\eps) - \e\mu_{H(z)}(f_\eps)|>\eps/2} \ll \expo{ -c_0pn^2}.
\end{equation}
Finally, we note that $H(z)$ has the form $H+H_0$ as in \Cref{lem:guze}, with $K=O(1/\sqrt{p})$, $H=H(0)$ and 
$$H_0= \begin{pmatrix} 0 & - zI\\ -z^*I & 0\end{pmatrix}.$$
Applying that lemma yields \eqref{goal4}, and hence \eqref{goal0_ch5}, which completes the proof.


\section{Bounds on small singular values}		\label{sec:redux}

In this section we deduce Propositions \ref{prop:least} and \ref{prop:small} from two general results on the small singular values of random matrices with independent entries: results of Rudelson--Zeitouni \cite{RuZe} and the author \cite{Cook:ssv} on the smallest singular value, and a new Wegner-type bound on ``moderately small" singular values. 

We begin by recalling some definitions and notation from \cite{Cook:ssv}.
The following allows us to quantify the dependence of our bounds on the distribution of the matrix entries.

\begin{definition}[Spread random variable]		\label{def:spread}
Let $\xi$ be a complex random variable and let $\kapa\ge 1$. 
We say that $\xi$ is \emph{$\kapa$-spread} if 
\begin{equation}	\label{def:kappa0}
\var \big[\,\xi\un(|\xi-\e\xi|\le \kapa)\,\big] \ge \frac1\kapa.
\end{equation}
\end{definition}

\begin{remark}	\label{rmk:kappap}
It follows from the monotone convergence theorem that any random variable $\xi$ with non-zero second moment is $\kapa$-spread for some $\kapa<\infty$.
Furthermore, if $\xi$ is centered with unit variance and finite $p$th moment $\mu_p$ for some $p>2$, then it is routine to verify that $\xi$ is $\kapa$-spread with $\kapa =3 (3\mu_p^p)^{1/(p-2)}$, say.
\end{remark}


Our results on small singular values assume the matrix $\Sig=(\sig_{ij})$ of standard deviations satisfies a certain expansion-type condition originally formulated in \cite{RuZe}. To state it we will need some graph-theoretic notation.
To a non-negative $n\times m$ matrix $\Sig=(\sig_{ij})$ we associate a bipartite graph $\Gamma_\Sig = ([n],[m], E_\Sig)$, with $(i,j)\in E_\Sig$ if and only if $\sig_{ij}>0$. 
For a row index $i\in [n]$ we denote by
\begin{equation}	\label{def:nbhd}
\mN_\Sig(i) = \set{j\in [m]: \sig_{ij}>0}
\end{equation}
its neighborhood in $\Gamma_\Sig$. 
Thus, the neighborhood of a column index $j\in [m]$ is denoted $\mN_{\Sig^\tran}(j)$.
For $I\subset [n]$ and $\delta\in (0,1)$, define the set of \emph{$\delta$-broadly connected} neighbors of $I$ as
\begin{equation}	\label{def:broadnbr}
\mN_{\Sig}^{(\delta)}(I)= \{ j\in [m]: |\mN_{\Sig^\tran}(j)\cap I|\ge \delta |I|\}.
\end{equation}
Given sets of row and column indices $I\subset [n], J\subset[m]$, we define the associated \emph{edge count}
\begin{equation}	\label{def:edges}
e_\Sig(I,J) := |\{(i,j)\in [n]\times [m]: \sig_{ij}>0\}|.
\end{equation}
We will generally work with the graph that only puts an edge $(i,j)$ when $\sig_{ij}$ exceeds some fixed cutoff parameter $\ha>0$.
Thus, we denote by
\begin{equation}	\label{Aha}
\Sig(\ha) = (\sig_{ij}1_{\sig_{ij}\ge \ha})
\end{equation}
the matrix which thresholds out entries smaller than $\ha$.

\begin{definition}[Random matrix with broadly connected profile]	\label{def:broadprofile}
Let $\Sig=(\sig_{ij})$ and $Z=(\be_{ij})$ be deterministic $n\times m$ matrices with $\sig_{ij}\in [0,1]$ and $\be_{ij}\in \C$ for all $i,j$.
We assume that $\Sig$ satisfies the following expansion property: for some $\ha,\delta,\nu\in (0,1)$ we have
\begin{enumerate}[(1)]
\item $|\mN_{\Sig(\ha)}(i)| \ge \delta m$ for all $i\in [n]$;\vspace{.2cm}
\item $|\mN_{\Sig(\ha)^\tran}(j)|\ge \delta n$ for all $j\in [m]$;\vspace{.2cm}
\item $|\mN_{\Sig(\ha)^\tran}^{(\delta)}(J)| \ge \min(n,(1+\nu)|J|)$ for all $J\subset [m]$.
\end{enumerate}
We say that a matrix $\Sig$ satisfying conditions (1)--(3) is \emph{$(\ha,\delta,\nu)$-broadly connected}.
Let $X=(\xi_{ij})$ be an $n\times m$ matrix with independent entries, all identically distributed to a complex random variable $\xi$ with mean zero and variance one.
Put 
\begin{equation}	\label{Mdef}
\M=\Sig\circ X+Z = (\sig_{ij}\xi_{ij}+ \be_{ij})_{i,j=1}^n
\end{equation}
where $\circ$ denotes the matrix Hadamard product.
We refer to $\Sig$ as the \emph{(standard deviation) profile} for $M$. 
Without loss of generality, we assume throughout that $\xi$ is $\kapa$-spread for some fixed $\kapa\ge1$.
\end{definition}

The following proposition collects known bounds on the smallest singular value for random matrices with a broadly connected profile.

\begin{proposition}[Smallest singular value: broadly connected profile]		\label{prop:broad}
Let $\M=\Sig\circ X+Z$ be an $n\times n$ matrix as in Definition \ref{def:broadprofile}.
Let $K_0\ge 1$. 
For any $t\ge 0$,
\begin{equation}	\label{broad:poly}
\pro{ s_n(\M) \le \frac{t}{\sqrt{n}}, \, \|\M\| \le K_0\sqrt{n} } \ll t + \frac{1}{\sqrt{n}}
\end{equation}
where the implied constant depends only on $K_0,\delta,\nu,\ha$ and $\kapa$.
If we further assume the entries of $X$ are standard real Gaussians, then
\begin{equation}	\label{broad:gaussian}
\pro{ s_n(\M) \le \frac{t}{\sqrt{n}}, \, \|\M\| \le K_0\sqrt{n} } \ll t + e^{-cn}
\end{equation}
where the implied constant and $c$ depend only on $K_0,\delta,\nu$ and $\ha$.
\end{proposition}

\begin{proof}
\eqref{broad:poly} follows from \cite[Theorem 1.13]{Cook:ssv}, while \eqref{broad:gaussian} follows from \cite[Theorem 2.3]{RuZe} (together with \eqref{lem:opnorm} and the triangle inequality for the operator norm). 
\end{proof}

We will prove the following theorem in \Cref{sec:wegner}.

\begin{theorem}[Wegner-type estimate: broadly connected profile]	 \label{thm:broad_small}
Let $M=\Sig\circ X+Z$ be an $n\times n$ matrix as in \Cref{def:broadprofile}.
Let $K_0\ge 1$ and $\eps\in (0,1)$.
There are constants $a_0,a_1,a_2>0$ depending (polynomially) on $\kapa,\ha,\delta,\nu,K_0$ such that
\begin{equation}
\pro{ \,  \exists k\in [n^{2\eps},a_1n]: s_{n-k}(M) < a_2 \frac{k}{\sqrt{n}} \;,\; \|M\|\le K_0\sqrt{n}\;} \ll \expo{-a_0n^\eps}
\end{equation}
where the implied constant depends on $\kapa,\ha,\delta,\nu,K_0$ and $\eps$.
\end{theorem}

To apply these results to the case that 
$\Sig=A_n$
we show the following:

\begin{proposition}[Random regular digraphs are broadly connected]	\label{prop:discrep}
Let $p\in (0,1)$ and let $A$ be a uniform random element of $\mA_{n,\lf p n\rf}$.
Then $A$ is $(1,p/2,p/8)$-broadly connected with probability $1-O_{p,K}(n^{-K})$ for any $K>0$.
\end{proposition}

To prove this we need the following result on edge counts in random regular digraphs, which is an immediate consequence of Corollary 1.9 in \cite{Cook:discrep}. (In the present setting the probability bound below can easily be improved, but it is sufficient for our purposes.) 

\begin{lemma}[No sparse patches]	\label{lem:discrep}
For $\eps_0>0$ let $\good(\eps_0)$ be the event that for all $I,J\subset[n]$ with $|I|,|J|\ge \eps_0n$, $e_A(I,J)\ge \frac12p|I||J|$.
For all $\eps_0>0$ we have
\[
\pro{ \good(\eps_0)} = 1-O_{p,\eps_0}(e^{-c\log^2n}).
\]
\end{lemma}

\begin{proof}[Proof of \Cref{prop:discrep}]
We may assume that $n$ is sufficiently large depending on $p$. 
$A$ clearly satisfies conditions (1) and (2) in \Cref{def:broadprofile} with $\delta = p/2$ since every row and column has support $\lf np\rf$. 
Now we verify condition (3).
Let $J\subset[n]$, and abbreviate $I_0=I_0(J) = \mN^{(p/2)}_{A^\tran}(J)$. 
Since each column of $A$ has support $\lf np\rf$, we have
\begin{align*}
\lf np\rf |J| &= \sum_{i=1}^n \sum_{j\in J} a_{ij}\\
&= \sum_{i\in I_0} |\mN(i)\cap J| + \sum_{i\notin I_0} |\mN(i)\cap J|\\
&< |I_0||J| + (p/2)n|J|
\end{align*}
which rearranges to 
\[
|I_0| > \frac12 pn - O(1) \gg pn.
\]
Thus, condition (3) is satisfied (deterministically) for any $J$ of size at most $cpn$ for a sufficiently small absolute constant $c>0$.
On the other hand, if $|J|>n(1-p/4)$ then we must have $|\mN(i)\cap J| \ge (p/2)n$ for all $i\in [n]$, so condition (3) also holds deterministically for any $J$ of size at least $n(1-p/4)$. 
So we may assume $cpn \le |J| \le (1-p/4)n$. 
In particular, 
\begin{equation}	\label{circ:Jbounds}
(1+2\nu)|J| = (1+p/4)|J|\le n
\end{equation}
Suppose that $|I_0(J)|< (1+\nu)|J|$.
Then 
\[
|I_0^c| >n-(1+\nu)|J| = \nu |J| + n-(1+2\nu)|J|\ge \nu |J|
\]
and
\[
e_A(I_0^c,J) = \sum_{i\in I_0^c} |\mN(i)\cap J| <\frac12p|I_0^c||J|.
\]
Thus, on the event that $|I_0(J)|<(1+\nu)|J|$ for some $J\subset[n]$, there exists $I\subset[n]$ such that $|I|\gg p|J|$ and $e_A(I,J)<\frac12p|I||J|$. 
But this means the event $\good(c'p^2)$ holds for some constant $c'>0$ sufficiently small, and the result follows from \Cref{lem:discrep}.
\end{proof}

Now we deduce Propositions \ref{prop:least} and \ref{prop:small} from \Cref{prop:broad} and \Cref{thm:broad_small}.
Fix $p\in (0,1)$, $z\in \C$, let $A,X,Y$ be as in \Cref{prop:least}, and write $Z:= -z\sqrt{np}I$, $M=Y+ Z$.
Conditional on any realization of $A$, by \Cref{lem:opnorm} we may restrict to the event $\{\|M\|\le K_0\sqrt{n}\}$ for some $K_0=O_{p,z}(1)$. 
By \Cref{prop:discrep} we may further condition on $A$ lying in the event that it is $(p/2,p/8)$-broadly connected.
Propositions \ref{prop:least} and \ref{prop:small} now follow from \Cref{prop:broad} and \Cref{thm:broad_small} with $\Sig=A$.

\section{Control of moderately small singular values}		\label{sec:wegner}

In this section we prove \Cref{thm:broad_small}.
While the traditional approach to obtaining such estimates (going back to the work of Bai \cite{Bai97}) is based on quantitative control on the rate of convergence of Stieltjes transforms, in \cite{TaVu:esd} Tao and Vu introduced a simpler and more direct geometric argument. The first key element of their approach is the so-called ``inverse second moment identity" (\eqref{i2m} below), which relates the size of the singular values of a rectangular matrix to the distances between rows and the span of the remaining rows. 

The second key element is a high probability lower bound for the distance between a random row vector $R$ and a fixed subspace $W$ of moderately large codimension $k$ (in our case of size $k\sim n^\eps$).
This distance can be expressed in the form $\|P_{W^\perp}R\|$, where $P_{W^\perp}$ is the matrix for projection to the orthogonal complement of $W$. 
If $R$ has i.i.d.\ centered components with unit variance, then
\begin{equation}	\label{e.iidcase}
\e \dist(R,W)^2 = \e \|P_{W^\perp}R\|^2 = \e R^\tran P_{W^\perp} R = \tr(P_{W^\perp}) = k.
\end{equation}
A high probability lower bound on $\dist(R,W)$ is then deduced from concentration of measure -- in this case Talagrand's isoperimetric inequality (after a truncation). 

The main new difficulty for proving \Cref{thm:broad_small} stems from the fact that the entries of $M$ are not identically distributed, in which case the computation \eqref{e.iidcase} breaks down, and in general the expected distance between a row and a fixed subspace can be quite small.
However, from a result in \cite{Cook:ssv} (\Cref{prop:slight} below), it turns out that under the broad connectivity assumption the orthogonal complement of a large number of rows is in ``generic position" in a certain sense. Specifically, any unit normal vector for a large number of rows is highly \emph{incompressible} (see \eqref{def:incompr} below). This can be used to obtain a lower bound on $\e \dist(R,W)^2$ that only loses a constant factor from the identity \eqref{e.iidcase}. (Observe that in the i.i.d.\ case the only information we used about $W$ was its dimension.)

We turn to the details. The following reduces our task to obtaining lower bounds on the distance between a fixed row and the span of most of the other rows.

\begin{lemma}	\label{lem:distlb}
Let $M\in \mM_n(\C)$ and $0\le k\le n-1$.
Put $m=n-\lceil k/2\rceil$. Denote the rows of $M$ by $R_1,\dots, R_n$, and for $i\in[m]$ denote 
\begin{equation}	\label{def:Rminus}
R_{-i}:= \Span \big\{R_j: j\in [m]\setminus \{i\}\big\}.
\end{equation}
We have
\begin{equation}	\label{lb:distance}
s_{n-k}(M) \gg \sqrt{\frac{k}{n}} \min_{i\in [m]} \dist(R_i,R_{-i}).
\end{equation}
\end{lemma}

\begin{proof}
We follow the argument from \cite{TaVu:esd}. 
Denote $M'=M_{[m]\times[n]}$, the matrix obtained by removing the last $\lceil k/2\rceil$ rows from $M$.
By the Cauchy interlacing law,
\begin{equation}	\label{cil.moderate}
s_{n-k}(M)\ge s_{n-k}(M').
\end{equation}
On the other hand, from the inverse second moment identity (cf.\ \cite[Lemma A.4]{TaVu:esd}) we have
\begin{equation}	\label{i2m}
\sum_{i=1}^n s_i(M')^{-2} = \sum_{i=1}^M \dist(R_i,R_{-i})^{-2}
\end{equation}
and so
\begin{align*}
m \Big( \min_{i\in [m]} \dist(R_i,R_{-i})\Big)^{-2} &\ge \sum_{i=1}^m \dist(R_i,R_{-i})^{-2} \\
&\ge \sum_{i=n-k}^{n-\lceil k/2\rceil} s_i(M')^{-2} \\
&\ge \frac{k}{2}s_{n-k}(M')^{-2}.
\end{align*}
\eqref{lb:distance} now follows from the above and \eqref{cil.moderate} (noting that $m\ge n/2$). 
\end{proof}

Our next task is to show how the distances $\dist(R_i,R_{-i})$ can be controlled from below if we have certain structural information on the normal vectors of $R_{-i}$. 
We recall the following definitions from \cite{RuZe,Cook:ssv}.
For $m\ge1$ and $\theta,\rho\in (0,1)$ we define the set of \emph{compressible vectors}
\begin{equation}	\label{def:compr}
\Comp_m(\theta,\rho) := S^{m-1}\cap\bigcup_{J\in {[m]\choose  \theta m}} (\C^J)_\rho
\end{equation}
and the complementary set of \emph{incompressible vectors}
\begin{equation}	\label{def:incompr}
\Incomp_m(\theta,\rho) := S^{m-1}\setminus \Comp(\theta,\rho)
\end{equation}
(recall our notational conventions from \Cref{sec:notation}).
That is, $\Comp_m(\theta,\rho)$ is the set of unit vectors within (Euclidean) distance $\rho$ of a vector supported on at most $\theta m$ coordinates.

\begin{lemma}[Distance of a random vector to an incompressible subspace]	\label{lem:dist}
Let $\xi$ be a centered complex-valued random variable with unit variance, and let $X=(\xi_1,\dots, \xi_n)$ be a vector of i.i.d.\ copies of $\xi$.
Let $\sig=(\sigma_1,\dots, \sigma_n)\in [0,1]^n$, and put $R=X\circ \sig=(\sigma_j\xi_j)_{j=1}^n$.
Suppose that for some $\delta,\ha\in (0,1)$ we have 
\[
|L|:=|\{j\in [n]:\sigma_j\ge \ha\}|\ge \delta n.
\]
Let $\eps\in (0,1)$ and let $V\subset\C^n$ be a subspace of dimension $n-k$.
Suppose that for some $\rho>0$ and any unit vector $u\in V^\perp$ we have $u\in \Incomp_n((1-\frac\delta2),\rho)$.
Then for any fixed $v\in \C^n$ we have
\begin{equation}
\pro{ \dist(R+v,V)\le c\rho\sqrt{\delta k}} =O_\eps(\exp(-c\ha^2\rho^2\delta k/n^{\eps}))
\end{equation}
if 
\begin{equation}	\label{dist:krange}
\frac{C n^{\eps}}{\ha^2\rho^2 \delta} \le k\le n-1
\end{equation}
for a sufficiently large constant $C>0$.
\end{lemma}

\begin{proof}
Fix $\eps\in (0,1)$.
We may assume $n$ is sufficiently large depending on $\eps$.
By replacing $V$ with $\Span(V,v)$ and $k$ with $k-1$ we may take $v=0$.

Towards an application of concentration of measure, we first perform a truncation.
By Chebyshev's inequality, for all $j\in \pr(|\xi_j|\ge n^{\eps/2}) \le n^{-\eps}$.
It follows from Hoeffding's inequality that 
\begin{equation}	\label{dist:Hoeffding}
\pro{ |\{j\in [n]: |\xi_j|\le n^{\eps/2}\}|\ge n - n^{1-\eps/2} } \ge 1-\exp(-cn^{1-\eps}).
\end{equation}
Put $m=\lceil n-n^{1-\eps/2}\rceil$, and for $J\in {[n]\choose m}$ denote the event
\begin{equation}
\event_J:= \{ |\xi_j|\le n^{\eps/2}\; \forall j\in J\}.
\end{equation}
It suffices to obtain control of the lower tail of $\dist(R,V)$ conditional on $\event_J$ that is uniform in the choice $J$.
For $j\in[n]$ let
\[
\lambda:= \e\big(\xi_j\,\big|\, |\xi_j|\le n^{\eps/2}\big),\quad\quad \xi_j' :=\xi_j-\lambda
\]
and write $R'=(\xi_1',\dots,\xi_n')$.
We have
\[
\tau^2:= \e\big(|\xi_j'|^2\, \big|\, |\xi_j|\le n^{\eps/2}\big) \ge1/2.
\]

Fix $J\in {[n]\choose m}$.
Condition on a realization of $\{\xi_j\}_{j\notin J}$, and write $\pr_J(\,\cdot\,)$ and $\e_J(\,\cdot\,)$ for probability and expectation conditional on $\{\xi_j\}_{j\notin J}$. 
Let $W=\Span( V, R_{[n]\setminus J}, \lambda \sig_J)$; note that $W$ is deterministic under the conditioning on $\{\xi_j\}_{j\notin J}$. Then $\dim(W)\le \dim(V)+2$ and 
\begin{equation}	\label{dist.prime}
\dist(R,V) \ge \dist(R',W).
\end{equation}
Note that $\dist(R',W)=\|P_{W^\perp}R\|$, where $P_{W^\perp}$ is the orthogonal projection to $W^\perp$. 
Letting $u^1,\dots, u^{k-2}$ be an arbitrary set of orthonormal vectors in $W^\perp$, we have
\begin{equation}
\e_J \dist(R',W)^2 \ge \sum_{i=1}^{k-2} \e |R'\cdot u^i|^2 = \tau^2\sum_{i=1}^{k-2} \sum_{j=1}^n |u^i_j\sigma_j|^2
\ge \frac12\ha^2 \sum_{i=1}^{k-2} \|(u^i)_L\|^2.
\end{equation}
For each $1\le i\le k-2$, since $u^i\in \Incomp_n((1-\frac\delta2),\rho)$ there is a set $J_i\subset[n]$ with $|J_i|\ge (1-\frac{\delta}2)n$ and $|u^i_j|\ge \rho/\sqrt{n}$ for all $j\in J_i$. Indeed, if this were not the case then the projection of $u^i$ to its largest $(1-\frac\delta2)n$ coordinates would be a $(1-\frac\delta2)n$-sparse vector lying a distance at most $\rho$ away from $u^i$, which contradicts $u^i\in \Incomp_n((1-\frac\delta2),\rho)$.
Thus, we have $\|(u^i)_{L}\|^2\ge \|(u^i)_{L\cap J_i}\|^2 \ge \frac12\delta \rho^2$.
Hence,
\begin{equation}	\label{lb:eJd}
\e_J\dist(R',W)^2 \ge c\delta\rho^2\ha^2k
\end{equation}
for some absolute constant $c>0$.

It remains to obtain a lower tail bound for $\dist(R',W)^2$.
Note that $v\mapsto \dist(\,\cdot \,,W)$ is a convex 1-Lipschitz function on $\C^J$.
Since the components of $R'$ are bounded in magnitude by $2n^{\eps/2}$, by Talagrand's concentration inequality \cite[Theorem 6.6]{Talagrand:newlook} (see also \cite[Corollary 4.4.11]{AGZ:book}) we have
\begin{equation}
\pr_J(|\dist(R',W)-d|\ge t) = O(\exp(-ct^2/n^{\eps}))
\end{equation}
for any $t\ge0$, where $d$ is any median for $\dist(R',W)$ conditional on $\{\xi_j\}_{j\notin J}$. 
In particular, $|\e_J\dist(R',W)-d|=O(n^{\eps/2})$ and $\var(\dist(R',W))=O(n^{\eps})$, so 
\[
d = \sqrt{\e \dist(R',W)^2} +O(n^{\eps/2}).
\]
Together with \eqref{lb:eJd} these estimates imply
\begin{equation}
\pr_J\Big(\dist(R',W)\le c\ha\rho\sqrt{\delta k}-Cn^{\eps/2}\,\Big) = O\big(\exp(-c\ha^2\rho^2\delta k/n^{\eps})\big)
\end{equation}
for some absolute constants $C,c>0$.
Undoing the conditioning on $\{\xi_j\}_{j\notin J}$, from the above and \eqref{dist.prime} we have
\begin{equation}
\pr\big( \dist(R,V) \le c\ha\rho\sqrt{\delta k}-Cn^{\eps/2} \,\big| \, \event_J \big) =O\big(\exp(-c\ha^2\rho^2\delta k/n^{\eps})\big).
\end{equation}
Summing over $J$ and using \eqref{dist:Hoeffding}, we conclude
\begin{align*}
\pro{ \dist(R,V) \le c\ha\rho\sqrt{\delta k} - Cn^{\eps/2}} = O(\exp(-c\ha^2\rho^2\delta k/n^{\eps}))
\end{align*}
and the result follows from the lower bound in \eqref{dist:krange}.
\end{proof}

We will need the following technical result, which is Proposition 3.1 from \cite{Cook:ssv}. It shows that under the broad connectivity hypothesis, the matrix $M$ is well-invertible on compressible vectors, even after the removal of some rows.

\begin{proposition}[Compressible vectors]	\label{prop:slight}
Let $\M=\Sig\circ X+Z$ be as in Definition \ref{def:broadprofile} with $n/2\le m\le 2n$. 
Let $K_0\ge1$. 
There exist $\theta_0(\kapa,\ha,\delta,K_0)>0$ and $\rho(\kapa,\ha,\delta,\nu,K_0)>0$ such that the following holds. 
Assume
\begin{enumerate}[(1)]
\item $|\mN_{\Sig(\ha)^\tran}(j)|\ge \delta n$ for all $j\in [m]$; \vspace{.2cm}
\item $|\mN_{\Sig(\ha)^\tran}^{(\delta)}(J)| \ge \min((1+\nu)|J|,n)$ for all $J\subset[m]$ with $|J|\ge \theta_0m$.
\end{enumerate}
Then for any $0<\theta\le (1-\frac\delta4)\min(\frac{n}m,1)$, 
\begin{equation}
\pro{\|M\|\le K_0\sqrt{n}, \exists u\in \Comp_m(\theta,\rho): \|Mu\|\le \rho K_0\sqrt{n}}\ll \expo{-c_0\delta\ha^2n}
\end{equation}
where $c_{0}>0$ depends only on $\kapa$ and the implied constant depends only on $\kapa,\ha,\delta,\nu$ and $K_0$.
\end{proposition}

\begin{remark}
In \cite[Proposition 3.1]{Cook:ssv} is stated under a more general distributional hypothesis for $\xi$; see \cite[Lemma 2.5]{Cook:ssv}.
\end{remark}

\begin{proof}[Proof of \Cref{thm:broad_small}]
For the duration of the proof we restrict to the event $\{\|M\|\le K_0\sqrt{n}\}$.
We may assume that $n$ is sufficiently large depending on $\kapa,\eps,\ha,\delta,\nu$, and $K_0$.
By the union bound it suffices to show that for $a_2$ sufficiently small depending on $\kapa,\ha,\delta,\nu,K_0$,
\begin{equation}	\label{goal:small1}
\pro{ s_{n-k}(M) \le a_2k/\sqrt{n}} = O(n\exp(-a_0n^{\eps}))
\end{equation}
for arbitrary fixed $k\in [n^{2\eps},a_1n]$. 
By \Cref{lem:distlb} and another application of the union bound, after modifying $a_2$ by a constant factor it suffices to show
\begin{equation}	\label{goal:small2}
\pro{ \dist(R_i,V_I) \le a_2\sqrt{k}} = O(\exp(-a_0n^\eps))
\end{equation}
where $V_I=\Span\{R_i, i\in I\}$ for an arbitrary fixed subset $I\subset[n]$ with $|I|=n-\lf k/2\rf -1=:n'$ and arbitrary fixed $i\in [n]\setminus I$.

Fix such $I\subset[n]$ and $i\in [n]\setminus I$. 
Let $\theta_0$ be as in \Cref{prop:slight} and let $J\subset[n]$ with $|J|\ge \theta_0n$.
Denoting $\widetilde{\Sig}:= \Sig(\ha)_{I\times[n]}$, by the assumption that $\Sig$ is $(\ha,\delta,\nu)$-broadly connected we have
\[
|\mN_{\widetilde{\Sig}^\tran}^{(\delta)}(J)| \ge \min\big(n', (1+\nu)|J| - \lceil k/2\rceil -1\big) \ge \min\big(n',(1+\nu/2)|J|\big)
\]
taking $a_1\le c\nu\theta_0$ for a sufficiently small constant $c>0$.
Applying \Cref{prop:slight} to $M_{I\times [n]}$ (with $(n',n)$ in place of $(n,m)$ and $(\delta/2,\nu/2)$ in place of $(\delta,\nu)$) we have that except with probability $O(\exp(-\frac12c_0\delta\ha^2n))$,
\begin{equation}	\label{small:goodevent}
\|M_{I\times [n]}u\| \ge \rho K_0\sqrt{n}\quad \forall u\in \Comp_n\left(\left(1-\frac\delta4\right)\frac{n'}{n}, \rho\right)
\end{equation}
for some $\rho=\rho(\kapa,\ha,\delta,\nu,K_0)>0$.
In particular, taking $a_1$ smaller if necessary, we have 
\[
\left(1-\frac\delta4\right)\frac{n'}n\ge \left(1-\frac\delta4\right) \left(1-\frac{a_1}{2}\right)  \ge 1-\frac\delta2
\]
and so
\begin{equation}
S^{n-1}\cap V_I^\perp = \ker(M_{I\times[n]}) \subset \Incomp_n(1-\delta/2,\rho).
\end{equation}
We may condition on the event that \eqref{small:goodevent} holds.
Conditional on $M_{I\times [n]}$, we apply \Cref{lem:dist} with $R+v=R_i$ and $V=V_I$ to obtain \eqref{goal:small2} as desired.
\end{proof}

\bibliographystyle{alpha}
\bibliography{rrd_signed}

\newcommand{\etalchar}[1]{$^{#1}$}
\begin{thebibliography}{AGB{\etalchar{+}}15}

\bibitem[AC15]{AdCh}
Rados{\l}aw Adamczak and Djalil Chafa{\"{\i}}.
\newblock Circular law for random matrices with unconditional log-concave
  distribution.
\newblock {\em Commun. Contemp. Math.}, 17(4):1550020 (22 pages), 2015.

\bibitem[ACW16]{ACW:exchangeable}
Rados{\l}aw Adamczak, Djalil Chafa{\"{\i}}, and Pawe{\l} Wolff.
\newblock Circular law for random matrices with exchangeable entries.
\newblock {\em Random Structures Algorithms}, 48(3):454--479, 2016.

\bibitem[AGB{\etalchar{+}}15]{AGBTAM15:foodwebs}
S.~Allesina, J.~Grilli, G.~Barab{\'a}s, S.~Tang, J.~Aljadeff, and A.~Maritan.
\newblock Predicting the stability of large structured food webs.
\newblock {\em Nature communications}, 6(7842), 2015.

\bibitem[AGZ10]{AGZ:book}
Greg~W. Anderson, Alice Guionnet, and Ofer Zeitouni.
\newblock {\em An introduction to random matrices}, volume 118 of {\em
  Cambridge Studies in Advanced Mathematics}.
\newblock Cambridge University Press, Cambridge, 2010.

\bibitem[ARS15]{ARS:block}
Johnatan Aljadeff, David Renfrew, and Merav Stern.
\newblock Eigenvalues of block structured asymmetric random matrices.
\newblock {\em J. Math. Phys.}, 56(10):103502, 14, 2015.

\bibitem[Bai97]{Bai97}
Z.~D. Bai.
\newblock Circular law.
\newblock {\em Ann. Probab.}, 25(1):494--529, 1997.

\bibitem[BC12]{BoCh:survey}
Charles Bordenave and Djalil Chafa{\"{\i}}.
\newblock Around the circular law.
\newblock {\em Probab. Surv.}, 9:1--89, 2012.

\bibitem[BCC11]{BCC:heavy}
Charles Bordenave, Pietro Caputo, and Djalil Chafa{\"{\i}}.
\newblock Spectrum of non-{H}ermitian heavy tailed random matrices.
\newblock {\em Comm. Math. Phys.}, 307(2):513--560, 2011.

\bibitem[BCZ]{BCZ:perm}
Anirban Basak, Nicholas Cook, and Ofer Zeitouni.
\newblock Circular law for the sum of random permutation matrices.
\newblock Preprint at arXiv:1705.09053.

\bibitem[BD13]{BaDe}
Anirban Basak and Amir Dembo.
\newblock Limiting spectral distribution of sums of unitary and orthogonal
  matrices.
\newblock {\em Electron. Commun. Probab.}, 18:no. 69, 19, 2013.

\bibitem[BHY16]{BHY:km}
Roland Bauerschmidt, Jiaoyang Huang, and Horng-Tzer Yau.
\newblock Local kesten--mckay law for random regular graphs.
\newblock Preprint at arXiv:1609.09052, 09 2016.

\bibitem[BKY]{BKY}
Roland Bauerschmidt, Antti Knowles, and Horng-Tzer Yau.
\newblock Local semicircle law for random regular graphs.
\newblock Preprint available at arXiv:1503.08702.

\bibitem[BS10]{BaSi10:book}
Zhidong Bai and Jack~W. Silverstein.
\newblock {\em Spectral analysis of large dimensional random matrices}.
\newblock Springer Series in Statistics. Springer, New York, second edition,
  2010.

\bibitem[CM05]{CaMc}
E.~Rodney Canfield and Brendan~D. McKay.
\newblock Asymptotic enumeration of dense 0-1 matrices with equal row sums and
  equal column sums.
\newblock {\em Electron. J. Combin.}, 12:Research Paper 29, 31 pp.
  (electronic), 2005.

\bibitem[Cooa]{Cook:rrdcirc}
Nicholas Cook.
\newblock The circular law for random regular digraphs.
\newblock Preprint available at arXiv:1703.05839.

\bibitem[Coob]{Cook:ssv}
Nicholas~A. Cook.
\newblock Lower bounds for the smallest singular value of structured random
  matrices.
\newblock Preprint at arXiv:1608.07347.

\bibitem[Coo16]{Cook:discrep}
Nicholas~A. Cook.
\newblock Discrepancy properties for random regular digraphs.
\newblock {\em Random Structures Algorithms}, 50:23--58, 2016.
\newblock doi:10.1002/rsa.20643.

\bibitem[Coo17]{Cook:sing}
Nicholas~A. Cook.
\newblock On the singularity of adjacency matrices for random regular digraphs.
\newblock {\em Probability Theory and Related Fields}, 167(1):143--200, Feb
  2017.
\newblock doi:10.1007/s00440-015-0679-8.

\bibitem[DP12]{DuPa}
Ioana Dumitriu and Soumik Pal.
\newblock Sparse regular random graphs: spectral density and eigenvectors.
\newblock {\em Ann. Probab.}, 40(5):2197--2235, 2012.

\bibitem[DS07]{DoSi07a}
R.~Brent Dozier and Jack~W. Silverstein.
\newblock On the empirical distribution of eigenvalues of large dimensional
  information-plus-noise-type matrices.
\newblock {\em J. Multivariate Anal.}, 98(4):678--694, 2007.

\bibitem[Ede97]{Edelman:circ}
Alan Edelman.
\newblock The probability that a random real gaussian matrix has k real
  eigenvalues, related distributions, and the circular law.
\newblock {\em Journal of Multivariate Analysis}, 60(2):203--232, 1997.

\bibitem[Gir84]{Girko84}
V.~L. Girko.
\newblock The circular law.
\newblock {\em Teor. Veroyatnost. i Primenen.}, 29(4):669--679, 1984.

\bibitem[GT10]{GoTi:circ}
Friedrich G{{\"o}}tze and Alexander Tikhomirov.
\newblock The circular law for random matrices.
\newblock {\em Ann. Probab.}, 38(4):1444--1491, 2010.

\bibitem[GZ00]{GuZe}
A.~Guionnet and O.~Zeitouni.
\newblock Concentration of the spectral measure for large matrices.
\newblock {\em Electron. Comm. Probab.}, 5:119--136 (electronic), 2000.

\bibitem[HL00]{HaLa}
Uffe Haagerup and Flemming Larsen.
\newblock Brown's spectral distribution measure for {$R$}-diagonal elements in
  finite von {N}eumann algebras.
\newblock {\em J. Funct. Anal.}, 176(2):331--367, 2000.

\bibitem[Jan95]{Janson:contiguity}
Svante Janson.
\newblock Random regular graphs: asymptotic distributions and contiguity.
\newblock {\em Combin. Probab. Comput.}, 4(4):369--405, 1995.

\bibitem[Lat05]{Latala}
Rafa{\l} Lata{\l}a.
\newblock Some estimates of norms of random matrices.
\newblock {\em Proc. Amer. Math. Soc.}, 133(5):1273--1282 (electronic), 2005.

\bibitem[LLT{\etalchar{+}}17]{LLTTY}
Alexander~E. Litvak, Anna Lytova, Konstantin Tikhomirov, Nicole
  Tomczak-Jaegermann, and Pierre Youssef.
\newblock Adjacency matrices of random digraphs: singularity and
  anti-concentration.
\newblock {\em J. Math. Anal. Appl.}, 445(2):1447--1491, 2017.

\bibitem[May72]{May72}
Robert~M May.
\newblock Will a large complex system be stable?
\newblock {\em Nature}, 238:413--414, 1972.

\bibitem[McK81]{McKay:esd}
Brendan~D. McKay.
\newblock The expected eigenvalue distribution of a large regular graph.
\newblock {\em Linear Algebra Appl.}, 40:203--216, 1981.

\bibitem[Meh67]{Mehta}
M.~L. Mehta.
\newblock {\em Random matrices and the statistical theory of energy levels}.
\newblock Academic Press, New York-London, 1967.

\bibitem[Ngu14]{Nguyen:uds}
Hoi~H. Nguyen.
\newblock Random doubly stochastic matrices: the circular law.
\newblock {\em Ann. Probab.}, 42(3):1161--1196, 2014.

\bibitem[PZ10]{PaZh}
Guangming Pan and Wang Zhou.
\newblock Circular law, extreme singular values and potential theory.
\newblock {\em J. Multivariate Anal.}, 101(3):645--656, 2010.

\bibitem[RA06]{RaAb:neural}
Kanaka Rajan and LF~Abbott.
\newblock Eigenvalue spectra of random matrices for neural networks.
\newblock {\em Physical review letters}, 97(18):188104, 2006.

\bibitem[RZ16]{RuZe}
Mark Rudelson and Ofer Zeitouni.
\newblock Singular values of {G}aussian matrices and permanent estimators.
\newblock {\em Random Structures Algorithms}, 48(1):183--212, 2016.

\bibitem[SCS88]{SCS:chaos}
H.~Sompolinsky, A.~Crisanti, and H.-J. Sommers.
\newblock Chaos in random neural networks.
\newblock {\em Phys. Rev. Lett.}, 61(3):259--262, 1988.

\bibitem[Tal96]{Talagrand:newlook}
Michel Talagrand.
\newblock A new look at independence.
\newblock {\em Ann. Probab.}, 24(1):1--34, 1996.

\bibitem[TV08]{TaVu:circ}
Terence Tao and Van~H. Vu.
\newblock Random matrices: the circular law.
\newblock {\em Commun. Contemp. Math.}, 10(2):261--307, 2008.

\bibitem[TV10]{TaVu:esd}
Terence Tao and Van~H. Vu.
\newblock Random matrices: universality of {ESD}s and the circular law.
\newblock {\em Ann. Probab.}, 38(5):2023--2065, 2010.
\newblock With an appendix by Manjunath Krishnapur.

\bibitem[TVW13]{TVW}
Linh~V. Tran, Van~H. Vu, and Ke~Wang.
\newblock Sparse random graphs: eigenvalues and eigenvectors.
\newblock {\em Random Structures Algorithms}, 42(1):110--134, 2013.

\bibitem[Woo12]{Wood:sparse}
Philip~Matchett Wood.
\newblock Universality and the circular law for sparse random matrices.
\newblock {\em Ann. Appl. Probab.}, 22(3):1266--1300, 2012.

\end{thebibliography}

\end{document}